\documentclass[12pt]{article}
\usepackage{fullpage}
\usepackage{amssymb,calc,enumerate,booktabs}
\usepackage{amsmath}
\usepackage{amsthm}
\usepackage{etoolbox}
\usepackage[algoruled,linesnumbered,longend,nofillcomment,noline]{algorithm2e}
\usepackage{units}
\usepackage{tikz}
\usepackage{pgfplots}
\usepackage{pdflscape}
\usepackage[affil-sl]{authblk}
\newcommand{\cl}[1]{\mathcal{#1}}
\newcommand\minus{%
  \setbox0=\hbox{-}%
  \vcenter{%
    \hrule width\wd0 height \the\fontdimen8\textfont3%
  }%
}

\usetikzlibrary{positioning,shapes,shadows,arrows}
\makeatletter
\def\@maketitle{%
  \newpage
  \null
  \vskip 2em%
  \begin{center}%
  \let \footnote \thanks
    {\Large\bfseries \@title \par}%
    \vskip 1.5em%
    {\normalsize
      \lineskip .5em%
      \begin{tabular}[t]{c}%
        \@author
      \end{tabular}\par}%
    \vskip 1em%
    {\normalsize \@date}%
  \end{center}%
  \par
  \vskip 1.5em}
\makeatother
\newtheorem{theorem}{Theorem}

\newtheorem{remark}{Remark}
\newtheorem{lemma}{Lemma}

\makeatletter
\patchcmd{\@algocf@start}{%
  \begin{lrbox}{\algocf@algobox}%
}{%
  \rule{0.1\textwidth}{\z@}%
  \begin{lrbox}{\algocf@algobox}%
  \begin{minipage}{0.8\textwidth}%
}{}{}
\patchcmd{\@algocf@finish}{%
  \end{lrbox}%
}{%
  \end{minipage}%
  \end{lrbox}%
}{}{}
\makeatother
\title{The bipartite unconstrained 0-1 quadratic programming problem: polynomially solvable cases\footnote{This research work was supported by an NSERC  Discovery grant awarded to Abraham P Punnen. }}
\author {Abraham P. Punnen\thanks{apunnen@sfu.ca}}
\author {Piyashat Sripratak\thanks{psriprat@sfu.ca}}
\author {Daniel Karapetyan\thanks{daniel.karapetyan@gmail.com}}
\affil{Department of Mathematics, Simon Fraser University Surrey, Central
City, 250-13450 102nd AV, Surrey, British Columbia, V3T 0A3, Canada}

\date{}

\begin{document}

\maketitle

\begin{abstract}
\noindent We consider the bipartite unconstrained 0-1 quadratic programming problem (BQP01) which is a generalization of the well studied unconstrained 0-1 quadratic programming problem (QP01). BQP01 has numerous applications and the problem is known to be MAX SNP hard.  We show that if the rank of an associated $m\times n$ cost matrix $Q=(q_{ij})$ is fixed, then BQP01 can be solved in polynomial time.  When $Q$ is of rank one, we provide an $O(n\log n)$ algorithm  and this complexity reduces to  $O(n)$ with additional assumptions. Further, if $q_{ij}=a_i+b_j$ for some  $a_i$ and $b_j$, then BQP01 is shown to be solvable in $O(mn\log n)$ time. By restricting $m=O(\log n),$ we obtain yet another polynomially solvable case of BQP01  but the problem remains MAX SNP hard if $m=O(\sqrt[k]{n})$ for a fixed $k$. Finally, if the minimum number of rows and columns to be deleted from $Q$ to make the remaining matrix non-negative is $O(\log n)$ then we show that BQP01 polynomially solvable but it is NP-hard if this number is $O(\sqrt[k]{n})$ for any fixed $k$.  \\

\end{abstract}

\section{Introduction}
The unconstrained 0-1 quadratic programming problem (QP01) is to
\begin{align*}
&\mbox{Maximize } f(x) = x^TQ^{\prime}x + c{^{\prime}}x+c^{\prime}_0\\
&\mbox{Subject to } x \in \{0,1\}^n,
\end{align*}
where $Q^{\prime}$ is an $n\times n$ matrix, $c{^{\prime}}$ is a row vector in $R^n$, and $c^{\prime}_0$ is a constant. The problem QP01 is formulated and studied in various alternative formats. For example, without loss of generality, one can assume that $c_0$ is zero, $c{^{\prime}}$ is the zero vector and $Q$ is symmetric. If $c^{\prime}$ is allowed to be arbitrary, then the diagonal elements of $Q$ can be assumed to be zero, without loss of generality. In another variation, $x$ is assumed to be in $\{-1,1\}^n$ instead of $\{0,1\}^n$. However, this variation can be reduced to the 0-1 version using a linear transformation. QP01 has been studied extensively in literature because of its various applications and rich theoretical structure. For further details and additional  references on QP01 we refer to~\cite{bo,r2,bor,mc,fer,ham1,a3,yw}. The focus of this paper is on a problem closely related to QP01 which we call the {\it bipartite unconstrained 0-1 quadratic programming problem} (BQP01). The problem BQP01 can be defined as follows.
\begin{align*}
&\mbox{Maximize } f(x,y) = x^TQy + cx+dy+c_0\\
&\mbox{Subject to }  x \in \{0,1\}^m, y\in \{0,1\}^n,
\end{align*}
where $Q = (q_{ij})$ is an $m\times n$ matrix, $c=(c_1,c_2, \ldots ,c_m)$ is a row vector in $R^m$, $d=(d_1,d_2,\ldots ,d_n)$ is a row vector in $R^n$ and $c_0$ is a constant.  Without loss of generality, we assume that $m \leq n$. Let
\begin{equation}\label{eq0}
\bar{Q}=\left[
\begin{array}{c|c}
O_{m\times m} & Q\\ \hline
O_{n\times m} & O_{n\times n}
\end{array}\right],\;
 \bar{c}=\left[
\begin{array}{c|c}
c & d
\end{array}\right]
 \mbox{ and }   w^T=\left[
\begin{array}{c|c}
x & y
\end{array}\right],
\end{equation}  where $O_{n\times n}$, $O_{n\times m}$ and $O_{m\times m}$ are zero matrices.
Then BQP01 can be formulated as the QP01
\begin{align}\label{uqp1}
&\mbox{Maximize } f(w) = w^T\bar{Q}w + \bar{c}w+c_0\\ \label{uqp2}
&\mbox{Subject to } w \in \{0,1\}^{m+n}.
\end{align}
Thus, the exact and heuristic algorithms available to solve QP01 can be used directly to solve BQP01. However, such a transformation increases the problem size and could limit our ability in handling large scale problems. Further, BQP01 has additional structure that can be exploited to obtain efficient algorithms and derive interesting theoretical properties.\\

We can also show that QP01 is a special case of BQP01. Choose \begin{equation}\label{equa}Q=Q^{\prime}+2MI,\; c=\frac{1}{2}c^{\prime}-M\textbf{u} \mbox{ and } d=\frac{1}{2}c^{\prime}-M\textbf{u},\end{equation} where $I$ is an $n\times n$ identity matrix, $\textbf{u}\in R^n$ is an all one row vector and $M$ is a very large number. Then the resulting BQP01 solves  QP01 since the penalty parameter $M$ forces $x_i=y_i$ in an optimal solution of this modified BQP01. The transformation discussed in equation (\ref{equa}) is important since it provides additional flexibility in developing algorithms for QP01 through BQP01 formulations. This provides additional motivation to study BQP01.\\

Consider the {\it continuous relaxation} BQPC of BQP01 given by:
\begin{flalign*}
\mbox{ BQPC: \hspace{1cm}} &\mbox{Maximize }  x^TQy + cx+dy+c_0&\\
&\mbox{Subject to }  x \in [0,1]^m, y\in [0,1]^n,&
\end{flalign*}

Note that multilinear polynomials in binary variables attain their maximum at a vertex of the unit cube~\cite{r2,r1}. Thus, BQPC and BQP01 are equivalent. BQPC is a special case of the {\it bilinear programming problem} (BLP)~\cite{alt,kono,yaj} and hence BQP01 can also be solved using any algorithm for solving the BLP\@. However, by exploiting the special structure of BQPC, more efficient algorithms may be obtained. To the best of our knowledge, BQP01 has not been studied in literature from the point of view of bilinear programs by exploiting the inherent structure of the problem.

A graph theoretic interpretation of BQP01 can be given as follows. Let $V_1=\{1,2,\ldots ,m\}$ and $V_2=\{1,2,\ldots ,n\}$. Consider the bipartite graph $G=(V_1,V_2,E)$.   For each node $i\in V_1$ and $j\in V_2$, respective costs $c_i$ and $d_j$ are prescribed. Further, for each $(i,j)\in E,$ a cost $q_{ij}$ is given. Then the {\it maximum weight induced subgraph problem} on $G$ is to find $U_1 \subseteq V_1$, $U_2\subseteq V_2$ such that $\sum_{i\in U_1}c_i+\sum_{j\in U_2}d_j + \sum_{(i,j)\in E^{1,2}}q_{ij}$ is maximized, where $ E^{1,2}$ is the edge set of the subgraph of $G$ induced $U_1\cup U_2$. The maximum weight induced subgraph problem on $G$ is precisely a BQP01, where $q_{ij}=0$ if $(i,j)\notin E$.

   Many well known combinatorial optimization problems can be modeled as BQP01. Consider the bipartite graph $G=(V_1,V_2,E)$ with $w_{ij} > 0$ being the weight of the edge $(i,j)\in E$.  Then the {\it maximum weight biclique problem} (MWBP)~\cite{amb,tan} is to find a biclique in $G$ of maximum total edge-weight.  Define $$q_{ij} = \begin{cases} w_{ij} &\mbox{if } (i,j)\in E \\
-M & \mbox{otherwise }, \end{cases} $$
where $M$ is a large positive number. Set $c$ and $d$ as zero vectors and also set $c_0=0$. Then the BQP01 with this choice of $Q,c,d$ and $c_0$ solves the MWBP\@. This immediately shows that BQP01 is NP-hard and one can also establish some approximation hardness results with appropriate assumptions~\cite{amb,tan}. MWBP  has applications in data mining, clustering and bioinformatics~\cite{chang,tanay} which in turn become applications of BQP01.

 Another application of BQP01 arises in approximating a matrix by a rank-one binary matrix~\cite{gills,a2,a1,lu,shen}. For example, let $H=(h_{ij})$ be a given $m\times n$ matrix and we want to find an $m\times n$ matrix $A=(a_{ij}),$ where $a_{ij}=u_iv_j$ and $u_i, v_j\in \{0,1\},$ such that $\sum_{i=1}^m\sum_{j=1}^n(h_{ij}-u_iv_j)^2$ is minimized. The matrix $A$ is called a rank one approximation of $H$ and can be identified by solving the BQP01 with $q_{ij}=1-2h_{ij},$ $c_i=0,$  $d_j=0$  and $c_0=0$ for all $i\in I$ and $j\in J$.  Binary matrix factorization is an important topic in mining discrete patterns in binary data~\cite{lu,shen}. If $u_i$ and $v_j$ are required to be in $ \{-1,1\}$  then also the resulting approximation problem can be formulated as a BQP01.
 The maximum cut problem on a bipartite graph can be formulated as a BQP01 and this gives yet another application of the model.  BQP01 can also be used to find approximations to the cut-norm of a matrix~\cite{alon,rs}.
When the matrix $Q^{\prime}$ is positive semidefinite, QP01 can be solved directly as a BQP01.  This follows from a corresponding result from BLP~\cite{kono} and the equivalence between BQPC and BQP01.

   To the best of our knowledge, BQP01 has not been thoroughly investigated in literature, especially from the point of view of polynomially solvable special cases. Some recent references on the problem considers theoretical analysis of approximation algorithms~\cite{ppk} and experimental analysis of heuristics~\cite{gtpk,kp}. The primary focus of this work is to identify polynomially solvable special cases of BQP01.

   We show that BQP01 can be solved in polynomial time if the rank of $Q$ is fixed. It may be noted that the corresponding version of QP01 is NP-hard even if $Q^{\prime}$ is of rank one~\cite{plh}. However, when rank of $Q^{\prime}$ is fixed and $Q^{\prime}$ is positive semidefinite with $c^{\prime}=0$, QP01 can be solved in polynomial time~\cite{all,fer,kar,mal}.  When $Q$ is of rank one, we show that $BQP01$ can be solved in $O(n\log n)$ time. In addition, we obtain an $O(n)$ algorithm for a special case of this rank-one problem. When $q_{ij}=a_i+b_j,$ we present an $O(mn\log n)$ algorithm.  Further, we observe that if $m=O(\log n)$ then BQP01 can be solved in polynomial time but if $m=O(\sqrt[k]{n})$ for a fixed $k$, the problem remains MAX SNP hard. Similarly, if the number of negative entries in $Q$ are $O(\log n)$ then BQP01 is solvable in polynomial time but if the number of negative entries are $O(\sqrt[k]{n})$ for a fixed $k$, the problem becomes NP-hard.

\section{Equivalent formulations}

QP01 is often presented in literature in various equivalent forms. Similar equivalent formulations can be obtained for BQP01 as well which are summarized in this section. Although many of the transformations discussed here follows almost directly from the corresponding transformation for QP01, we present them here for completness, comparison, and for notational convenience. Further, these transformations, albeit simple,  preserve the bipartite structure of the problem which allows us establishing interesting complexity results.

When the linear and constant terms from BQP01 are dropped, we get the homogeneous version of the problem defined as,
\begin{flalign*}
\mbox{BQP01(H):\hspace{1cm} }&\mbox{Maximize } f(x,y) = x^TQy& \\
&\mbox{Subject to } x \in \{0,1\}^m, y\in \{0,1\}^n.&
\end{flalign*}
In fact, the general BQP01 can be represented in the homogeneous form by introducing two additional variables. Let \begin{equation}\label{barq} \bar{Q}=\left[
\begin{array}{c|c}
Q & c^T\\ \hline
d & c_0
\end{array}\right].\end{equation} Then  BQP01 is equivalent to
\begin{flalign*}
\mbox{BQP01($\bar{H}$):\hspace{1cm} }&\mbox{Maximize } f(\bar{x},\bar{y}) = \bar{x}^T\bar{Q}\bar{y} &\\
&\mbox{Subject to } \bar{x} \in \{0,1\}^{m+1}, \bar{y}\in \{0,1\}^{n+1}, \bar{x}_{m+1} =1, \bar{y}_{n+1}=1,&
\end{flalign*}
where $\bar{x}_{m+1}$ and $\bar{y}_{n+1}$ are respectively the $(m+1)$th and $(n+1)$th components of $\bar{x}$ and $\bar{y}$. The restrictions $\bar{x}_{m+1} =1, \bar{y}_{n+1}=1$ in BQP01($\bar{H}$) can be dropped by replacing $c_0$ with a large number $M$ in $\bar{Q}$. Thus, we have
\begin{remark}\label{rm1}
A BQP01 can always be represented in an equivalent homogeneous form in the sense that the two problems have identical optimal solutions set.
\end{remark}
 Another important variation of BQP01 is obtained by restricting the variables to $1$ or $-1$ instead of $0$ or $1$. We call this the {\it cut version} of BQP01 and denote it by BQP-11. The problem BQP-11 can be stated as
\begin{flalign*}
\mbox{BQP-11:\hspace{1cm} }&\mbox{Maximize } \phi(x,y) = x^TQy+cx+dy+c_0 &\\
&\mbox{Subject to } x \in \{-1,1\}^m, y\in \{-1,1\}^n.&
\end{flalign*}
By dropping the linear and constant terms from BQP-11, we get the corresponding homogeneous version
\begin{flalign*}
\mbox{BQP-11(H):\hspace{1cm} }&\mbox{Maximize } \phi(x,y) = x^TQy& \\
&\mbox{Subject to } x \in \{-1,1\}^m, y\in \{-1,1\}^n.&
\end{flalign*}
As in the case of  BQP01, its cut version BQP-11 can also be represented in the corresponding homogeneous form by introducing two additional variables. Let $ \bar{Q}$ be defined as in equation (\ref{barq}) and  $\bar{x} \in \{-1,1\}^{m+1}$, $\bar{y}\in \{-1,1\}^{n+1}$.
 Then BQP-11 is equivalent to
\begin{flalign*}
\mbox{BQP-11($\bar{H}$):\hspace{1cm} }&\mbox{Maximize } \phi(\bar{x},\bar{y}) = \bar{x}^T\bar{Q}\bar{y}& \\
&\mbox{Subject to } \bar{x} \in \{-1,1\}^{m+1}, \bar{y}\in \{-1,1\}^{n+1}, \bar{x}_{m+1} =1, \bar{y}_{n+1}=1.&
\end{flalign*}
Note that the restriction $\bar{x}_{m+1} =1, \bar{y}_{n+1}=1$ in BQP-11($\bar{H}$) can be replaced by $\bar{x}_{m+1} =\bar{y}_{n+1}$ without affecting the optimal solution of BQP-11($\bar{H}$). This is possible because, for any solution $\bar{x},\bar{y}$ with $\bar{x}_{m+1} =-1, \bar{y}_{n+1}=-1$, the solution $-\bar{x},-\bar{y}$ satisfies $-\bar{x}_{m+1} =1, -\bar{y}_{n+1}=1$ and $\phi(\bar{x},\bar{y})=\phi(-\bar{x},-\bar{y})$. Further, the condition
$\bar{x}_{m+1} =\bar{y}_{n+1}$ can be dropped by replacing $c_0$ with a large number $M$ in $\bar{Q}$ resulting in a homogeneous version of BQP-11. Thus, we have
\begin{remark}\label{rm2}
A BQP-11 can always be represented in an equivalent homogeneous version in the sense that the two problems have identical optimal solutions set.
\end{remark}

Consider the linear transformation \begin{equation}\label{tr1}x=2w-e_m \mbox{  and } y=2z-e_n,\end{equation} where $e_m$ and $e_n$ are all-one vectors in $R^m$ and $R^n$, respectively. Using (\ref{tr1}), BQP-11 can be reduced to the BQP01
\begin{align*}
&\mbox{Maximize } f(w,z) = w^T\tilde{Q}z + \tilde{c}w+\tilde{d}z+\tilde{c}_0&\\
&\mbox{Subject to }  w \in \{0,1\}^m, z\in \{0,1\}^n,&
\end{align*}
where $\tilde{Q}=4Q$, $\tilde{c}=2(c-(Qe_n)^T)$, $\tilde{d}=2(d-e_m^TQ)$ and  $\tilde{c}_0=e_m^TQe_n-ce_m-de_n+c_0$.\\

Similarly, using the linear transformation
\begin{equation}\label{tr2}
x=\frac{1}{2}(w+e_m) \mbox{ and } y = \frac{1}{2}(z+e_n)
\end{equation}
BQP01 can be reduced to the BQP-11
\begin{align*}
&\mbox{Maximize } \phi(w,z) = w^T\hat{Q}z + \hat{c}w+\hat{d}z+\hat{c}_0&\\
&\mbox{Subject to }  w \in \{-1,1\}^m, z\in \{-1,1\}^n&
\end{align*}
where $\hat{Q}=\frac{1}{4}Q,$ $\hat{c}=\frac{1}{4}(Qe_n)^T+\frac{1}{2}c,$ $\hat{d}=\frac{1}{4}e^T_mQ+\frac{1}{2}d$ and $\hat{c}_0=\frac{1}{4}e^T_mQe_n+\frac{1}{2}ce_m+\frac{1}{2}de_n+c_0.$

It is well known that the maximum weight cut problem (MaxCut) on a general graph $G$ is equivalent to QP01~\cite{ham1}. In the maximum weight cut problem, if we restrict the graph $G$ to be bipartite, we get an instance of the {\it bipartite maximum weight cut problem} which is denoted by B-MaxCut. Indeed, viewing B-MaxCut as a general MaxCut problem yields an equivalent QP01.  Many approximation algorithms for MaxCut assume that the associated edge weights are nonnegative~\cite{gw}. However, for nonnegative edge weights, B-MaxCut is a trivial problem since the generic bipartition of the underlying graph gives the optimal cut. Thus, the topic that is more interesting is when the edge weights take positive as well as negative values. Developing approximation algorithms for BQP01 or for B-MaxCut is not the focus of this paper. However, we do want to indicate the equivalence between B-MaxCut and BQP01 to further understand the inherent complexity of the problems.

Let $G=(I,J,E)$ be a bipartite graph. Two vectors  $x\in \{-1,1\}^m$ and $y\in \{-1,1\}^n$ determine a cut $(U_1\cup U_2, H_1\cup H_2)$ in $G$ if and only if
$U_1=\{i\in I : x_i=1\},$ $ H_1=\{i\in I : x_i=-1\},$ $U_2=\{j\in J : y_j=1\},$ and $H_2=\{j\in J : y_j=-1\}.$  We call $(x,y)$ the {\it incidence vector} of the cut $(U_1\cup U_2, H_1\cup H_2)$. Let $q_{ij}$ be the weight of the edge $(i,j)$ in $G$. Then the value of the cut $(U_1\cup U_2, H_1\cup H_2)$ is given by
\begin{equation}
\sum_{i\in U_1, j\in H_2}q_{ij}+\sum_{i\in H_1,j\in U_2}q_{ij}=
\sum_{x_i=-y_j}q_{ij}
\end{equation}

\begin{theorem} The problems BQP01 and B-MaxCut are  equivalent in the sense that:
 \begin{enumerate}\vspace{-5pt}\item For any instance of BQP01, it is possible to construct a complete bipartite graph $G$ such that an optimal solution to the B-MaxCut problem on $G$ gives an optimal solution to BQP01.
\item For any instance of B-MaxCut on a bipartite graph $G=(I,J,E)$, it is possible to construct an instance of BQP01 with an $m\times n$ cost matrix $Q$ such that an optimal solution to the BQP01 gives an optimal solution to the B-MaxCut problem on $G$.
\end{enumerate}
\end{theorem}
\begin{proof}
 Since BQP01 is equivalent to a BQP-11(H) (see remarks \ref{rm1} and \ref{rm2}), without loss of generality we assume that BQP01 is given in the equivalent homogeneous cut form BQP-11(H). Now,
\begin{align*}
\phi(x,y)&=\sum_{ij}q_{ij}x_iy_j = \sum_{x_i=y_j}q_{ij}-\sum_{x_i=-y_j}q_{ij} =\sum_{ij}q_{ij} - 2\sum_{x_i=-y_j}q_{ij}
\end{align*}
Since $\sum_{ij}q_{ij}$ is a constant, maximizing $\phi(x,y)$ is equivalent to maximizing $- 2\sum_{\bar{x}_i=-\bar{y}_j}q_{ij}$. Thus, by solving the B-MaxCut problem on a complete bipartite graph $K_{m,n}$ with $-2q_{ij}$  chosen as the weight of the edge $(i,j)$, we can solve BQP-11(H).

To establish the second part of the theorem, we show that the B-MaxCut problem on $G$ with edge weights $c_{ij}$ for $(i,j)\in E$ can be solved as a BQP-11(H)\@. Let $(U_1\cup U_2, H_1\cup H_2)$ be a cut in $G$ and $(x,y)$  be the corresponding incidence vector. Let $g(U,H)= \sum_{x_i=-y_j}c_{ij}$ be the value of the cut $(U_1\cup U_2, H_1\cup H_2)$. Then it can be verified that
$g(U,H) = \frac{1}{2}\sum_{ij}c_{ij}-\frac{1}{2}\sum_{ij}c_{ij}x_iy_j$ and, hence, maximizing $g(U,H)$ is equivalent to solving a BQP-11(H) with cost matrix $q_{ij}=-\frac{1}{2}c_{ij}$.
\end{proof}

\section{Polynomially solvable cases}

In BQP01 if we restrict the variables to $1$ or $-1$ instead of $0$ or $1$, we get the {cut version} of the problem. BQP01 and its cut version can be reduced to each other using appropriate linear transformations. Alon and Naor~\cite{alon} showed that the cut version of BQP01 with no linear or constant terms is MAX SNP hard. As a consequence, BQP01 is also MAX SNP hard. Thus, let us  focus on polynomially solvable special cases of BQP01.

 Polynomially solvable special cases of QP01 have been investigated by many authors~\cite{all,cela,fer,plh,kar,li,mal,pr}. Since BQP01 can be formulated as a QP01 using a linear transformation, whenever the transformed data satisfies the known conditions for polynomial solvability of QP01, we can solve the corresponding BQP01 also in polynomial time. Thus, it is interesting  to focus on new polynomially solvable cases that exploit the special structure of BQP01.

\subsection{Fixed parameter problems}

Let us first consider a simple case where the number of rows of $G$ is fixed.

\begin{theorem}\label{thmp1} BQP01 can be solved in polynomial time if $m=O(\log n)$ but remains MAX SNP hard if $m=O(\sqrt[k]{n})$ for any constant $k$. \end{theorem}
\begin{proof} If $x$ is fixed to all possible values, the resulting objective function linear in $y$ and have a closed form solution of the type:
\begin{equation}\label{eq2}y^0_j = \begin{cases} 1 &\mbox{if } \sum_{i=1}^mq_{ij}x _i+d_j > 0 \\
0 & \mbox{otherwise }. \end{cases} \end{equation}
Since there are only $2^m$ choices for $x$, the polynomial solvability for $m=O(\log n)$ follows.
To establish that BQP01 is MAX SNP hard when $m=O(\sqrt[k]{n})$, consider an instance of BQP01 with data $Q,c,d$ and $c_0$. Let \begin{equation*}\label{eq0}
\bar{Q}=\left[
\begin{array}{c|c}
Q & O_{m\times n^k}\\ \hline
O_{(n-m)\times (n)} & O_{(n-m)\times n^k}
\end{array}\right],
 \bar{c}=\left[
\begin{array}{c|c}
c & \tilde{c}
\end{array}\right],
 \mbox{ and }   \bar{d}=\left[
\begin{array}{c|c}
d & \tilde{d}
\end{array}\right],
\end{equation*}  where $O_{m\times n^k}$, $O_{(n-m)\times (n)}$ and $ O_{(n-m)\times n^k}$ are zero matrices, $\tilde{c}$ is a zero row vector of dimension $n-m$ and $\tilde{d}$ is a zero row vector of dimension $n^k$. Now consider the instance of BQP01 with data $\bar{Q},\bar{c},\bar{d}$ and $c_0$. It can be verified that this new instance of BQP01 satisfies the condition of the theorem. Further, for every solution to the original BQP01, we can find a solution to the constructed instance of BQP01, where the corresponding objective function values are identical and viceversa. Since BQP01 is MAX SNP hard, the result follows.
\end{proof}
If $m$ is fixed, it can be verified that BQP01 can be solved in $O(n)$ time. Further, if $m=O(\log^k n)$ for fixed $k$, the complexity becomes quasi-polynomial.\\


Let $I\subseteq V_1=\{1,2,\ldots, m\}$ and $J\subseteq V_2=\{1,2,\ldots ,n\}$. We say that $I\cup J$ is a \textit{negative eliminator} of $Q$ if the matrix $Q^*$ obtained from $Q$ by deleting rows corresponding to $I$ and columns corresponding $J$ has only non-negative entries.
 A negative eliminator of smallest cardinality is called a minimum negative eliminator. Construct a bipartite graph $G_B=(V_1,V_2,E)$ where $(i,j)\in E$ if and only if $q_{ij} < 0$. Then a minimum negative eliminator of $Q$ is precisely a minimum vertex cover of $G_B$. Since the vertex cover problem on a bipartite graph can be solved in polynomial time, the minimum negative eliminator of $Q$ can be identified in polynomial time.
 Let $S^-=I^-\cup J^-$  be a minimum negative eliminator of $Q$.

\begin{theorem}\label{thfq} BQP01 can be solved in polynomial time if  $|S^-|$ is $ O(\log n)$ and the problem is NP-hard if $|S^-|$ is $O(\sqrt[k]{n})$ for some fixed $k$. \end{theorem}
\begin{proof}
Suppose $|S^-|$ is $ O(\log n)$.  Fixing the variables $x_i$ for $i\in I^-$ and $y_j$ for $j\in J^-$ at 0-1 values results in a reduced BQP01 with cost matrix have non-negative elements. Such a BQP01 can be solved as a minimum cut problem~\cite{r3}. Since there are at most $2^{|S^-|}$ ways to fix the variables associated with $I^-$ and $J^-$, BQP01 can be solved in polynomial time if $|S^-|=O(\log n)$.

The second part of the theorem can be proved by  reducing a general BQP01 to a BQP01 satisfying the conditions of the theorem. This can be achieved by increasing the number of  columns (or rows) of $Q$ to a sufficiently large number, yet polynomial for fixed $k$ and filling these columns (rows) with entries 0.
\end{proof}

It may be noted that Theorem~\ref{thfq} allows arbitrary $c$ and $d$.

\subsection{Fixed rank cost matrix}

QP01 is polynomially solvable if $Q^{\prime}$ is a symmetric positive semidefinite matrix~\cite{all,fer,kar,mal} with fixed rank and $c^{\prime}=0$. If $c^{\prime}$ is allowed to have arbitrary elements, the problem is NP-hard even if the rank of $Q^{\prime}$ is one~\cite{plh}.   We now show that BQP01 is solvable in polynomial time if rank of  $Q$ is fixed. No assumption is made on any other property of $Q$ such as positive semidefiniteness  and no restrictions on $c$ and $d$ are imposed. Our algorithm was inspired by~\cite{kky}.   First, let us prove some preliminary results.

Consider the multiparametric linear programming problem (MLP)
 \begin{align*}
&f(\lambda)=\max  cx &\\
&\text{Subject to: } Ax = \lambda&\\
 &\hspace{2.0cm}x \in [0,1]^m,&
\end{align*}

\noindent where $A$ is a $p\times m$ matrix of full row rank and $\lambda^T=(\lambda_1,\lambda_2,\ldots ,\lambda_p)\in R^p$. The $j^{\mbox{th}}$ column of the matrix $A$ is denoted by $A_j$. A partition $(\cl{B,L,U})$  of $I=\{1,2,\ldots ,m\}$ with $|\cl{B}|=p$ is referred to as a {\it basis structure} for MLP. For each basic feasible solution of MLP, a basis structure $(\cl{B,L,U})$ is associated, where $\cl{L}$ is the index set of nonbasic variables at the lower bound $0$,  $\cl{U}$ is the index set of non-basic variables at the upper bound $1$ and  $\cl{B}=\{B1,B2,\ldots, Bp\}$ is the index set of basic variables. Let $B$ be the $p\times p$ matrix with its $i^{\mbox{th}}$ column $A_{Bi}$. Then  the set $\cl{B}$ is called a basis set or simply a {\it basis} and $B$ is the associated {\it basis matrix}. A basis provides an implicit ordering of its elements. Thus, $x_{Bi}$ is called the $i^{\mbox{th}}$ basic variable with respect to $\cl{B}$.   Let $C_B=(c_{B1},c_{B2},\ldots ,c_{Bp})$. Then the basis structure $(\cl{B,L,U})$ is dual feasible if and only if~\cite{murty}
\begin{align}\label{redu}
C_BB^{-1}A_j-c_j \geq & 0 \mbox{ for } j\in \cl{L}\\ \label{red1}
C_BB^{-1}A_j-c_j \leq & 0 \mbox{ for } j\in \cl{U}.
\end{align}
Conditions (\ref{redu}) and (\ref{red1}) are also known as {\it reduced cost optimality conditions}~\cite{murty}. Let $(\cl{B,L,U})$ be a dual feasible basis structure for MLP\@. Then the basic solution of MLP corresponding to this basis structure is optimal for all $\lambda\in R^p$ satisfying
\begin{align}\label{feas}
\mathbf{0} \leq B^{-1}\lambda - B^{-1}A^u\textbf{v} \leq \mathbf{1},
\end{align}
where $A^u$ is the submatrix formed by columns of $A$ corresponding to the indices of $\cl{U}$, $\textbf{v}$ is a vector of size $|\cl{U}|$ with all entry equal to 1, $\mathbf{0}$ is the zero vector in $R^p$ and $\mathbf{1}$ is the vector in $R^p$ with all entries equal to 1. The polyhedral set represented by (\ref{feas}) is called the {\it characteristic region} of the basis structure $(\cl{B,L,U})$. A dual feasible basis structure $(\cl{B,L,U})$ is {\it dual non-degenerate} if (\ref{redu}) and (\ref{red1}) are satisfied as strict inequalities. If any of these inequalities is satisfied as an equality, then $(\cl{B,L,U})$ is a {\it dual degenerate} basis structure.

Let $S$ be the collection of all the extreme points of the characteristic regions associated with all dual feasible basis structures of MLP.
\begin{lemma}\label{lm7}
$|S| \leq {}^m\cl{C}_p 2^p$.
\end{lemma}
\begin{proof}
For simplicity, we assume that all dual feasible basis structures of MLP are dual non-degenerate. This is not a restriction since we can achieve this by an $\epsilon$-perturbation of the cost vector $c$ in MLP by an appropriately small $\epsilon$ and this change will not underestimate $|S|$. Thus, the inequalities (\ref{redu}) and (\ref{red1}) are satisfied as strict inequalities. Consequently, given $\cl{B}$, the choice of $\cl{L}$ and $\cl{U}$ is unique for any dual feasible basis structure, i.e., if $(\cl{B},\cl{L}_1,\cl{U}_1)$ and $(\cl{B},\cl{L}_2,\cl{U}_2)$ are two dual non-degenerate basis structures for MLP then at most one of them can  be dual feasible.    Thus, there exist at most ${}^m\cl{C}_p$ dual feasible and dual non-degenerate basis structures for MLP. The characteristic region of such a basis structure $(\cl{B,L,U})$ is defined by the inequalities (\ref{feas}). Each extreme point of this polyhedron is determined by the unique solution of $p$ tight inequalities from (\ref{feas}) that are satisfied as equalities. Since there are exactly $2^p$ choices for these tight inequalities, the result follows.
\end{proof}

Note that $f({\lambda})$ is a piecewise linear concave function when $\lambda \in R^p$~\cite{gal}. It is linear when $\lambda$ is restricted to a characteristic region associated with any dual feasible basis structure $(\cl{B,L,U})$.  These extreme points are called {\it breakpoints} of $f(\lambda)$. Thus, $f(\lambda)$ will have at most  ${}^m\cl{C}_p 2^p$ breakpoints.

Let $S(\cl{B,L,U})$ be the collection of all extreme points of the characteristic region associated with $\cl{(B,L,U)}$. Then, using inequalities (\ref{feas}), it can be verified that
\begin{align}\label{sblu}
S(\cl{B,L,U}) & = \{\lambda_{\tau} : \lambda_{\tau}=B\tau + A^u\textbf{v}, \tau \in \{0,1\}^p\}
\end{align}
where $\textbf{v}$ is the all-one vector of size $|\cl{U}|$. Let $\bar{S}(\cl{B,L,U})$ be the collection of all basic feasible solutions of MLP associated with the extreme points in $S(\cl{B,L,U})$.

\begin{theorem}\label{zeroone}$\bar{S}(\cl{B,L,U})\subseteq \{0,1\}^m$ and $|\bar{S}(\cl{B,L,U})|=2^p$.
\end{theorem}
\begin{proof}
Let $x(\lambda)\in \bar{S}(\cl{B,L,U})$ be a basic feasible solution for MLP corresponding to the extreme point $\lambda \in S(\cl{B,L,U})$ and $B$ be the basis matrix associate with $(\cl{B,L,U})$. Let $x^{\lambda}_B$ be the vector of basic variables of $x(\lambda)$. From (\ref{sblu}), $\lambda=B\tau + A^u\textbf{v}$ for some $\tau \in \{0,1\}^p$ and hence
\begin{align}\label{bv}
x^{\lambda}_B=B^{-1}\lambda - B^{-1}A^u\textbf{v}=\tau.
\end{align}
The non-basic variables of $x(\lambda)$ by definition take 0-1 values and hence $x(\lambda)\in \{0,1\}^m$. Thus $x(\lambda) \subseteq \bar{S}(\cl{B,L,U})$.  Since there are $2^p$ choices for $\tau$ in equation (\ref{bv}) and nonbasic variables in $\cl{L}$ and $\cl{U}$ are fixed for a given $\cl{(B,L,U)}$ (independent of the choice of $\tau$), we have $|\bar{S}(\cl{B,L,U})|=2^p$.
 \end{proof}

 Let us now consider BQP01 where rank of $Q$ is $p$, a fixed number. We assume that $Q$ is given in the rank-$p$ factorized form. i.e. $Q=AB$ where $A$ is an $m\times p$ matrix and $B$ is a $p\times n$ matrix. Such a factorization can easily be constructed from the reduced row echelon form $Q^R$ of $Q$ by choosing $A$ as the matrix obtained from $Q$ by deleting non-pivot columns and $B$ as the matrix obtained from $Q^R$ by removing the zero rows.  Let $a^k=(a^k_1,a^k_2,\ldots ,a^k_m)$ be the $k^{\mbox{th}}$ column of $A$ and $b^k=(b^k_1,b^k_2,\ldots ,b^k_n)$ be the $k^{\mbox{th}}$ row of $B$.   Since BQP01 is equivalent to BQPC, this problem can be stated as
  \begin{flalign*}
\mbox{BQPC(p):\hspace{1cm} }&\mbox{ Maximize } \sum_{k=1}^pa^kxb^ky+cx+dy &\\
&\text{Subject to: } x \in [0,1]^m, y\in [0,1]^n&.
\end{flalign*}
 Consider the multiparametric linear program~\cite{gal}
 \begin{flalign*}
\mbox{MLP1:\hspace{1cm} }&h_1(\lambda)=\max  cx &\\
&\text{Subject to: } a^kx = \lambda_k \mbox{ for } k=1,2,\ldots ,p&\\
 &\hspace{2.0cm}x \in [0,1]^m,&
\end{flalign*}
\noindent where $\lambda=(\lambda_1,\lambda_2,\ldots ,\lambda_p)$. Then $h_1(\lambda)$ is a piecewise linear concave function~\cite{gal}. Let $S_1$ be the set of breakpoints of $h_1(\lambda)$ and $x(\lambda)$ be an optimal basic feasible solution of MLP1 at $\lambda\in S_1$. By Theorem~\ref{zeroone}, $x(\lambda)\in \{0,1\}^m$. Let $y(\lambda)=(y_1(\lambda),y_2(\lambda), \ldots ,y_n(\lambda))$ be an optimal solution to BQPC(p) when $x$ is restricted to $x(\lambda)$. Then it can be verified that $y(\lambda)$ satisfies
\begin{align}\label{opty}
y_j(\lambda) = \begin{cases} 1 & \mbox{ if } d_j+\sum_{k=1}^pb^k_j\sum_{i=1}^ma^k_ix_i(\lambda) > 0 \\
0 & \mbox{otherwise }.
\end{cases}
\end{align}

\begin{theorem}\label{opt}There exists an optimal solution to BQPC(p) amongst the solutions $\{(x(\lambda),y(\lambda)): \lambda \in S_1\}$.\end{theorem}
\begin{proof}
BQPC(p) is equivalent to the bilinear program
\begin{flalign*}
\mbox{BLP1: \hspace{1cm}}&\mbox{ Maximize }\sum_{k=1}^p\lambda_kb^ky + cx +dy&\\
&\text{Subject to: } a^kx = \lambda_k \mbox{ for } k=1,2,\ldots ,p&\\
 &\hspace{2.0cm}x \in [0,1]^m, y\in [0,1]^n,\lambda=(\lambda_1,\lambda_2,\ldots ,\lambda_p)^T\in R^p.&
\end{flalign*}
Let $h(\lambda)$ be the optimal objective function value of BLP1 where $\lambda$ is fixed. Now, $h(\lambda)$ can be decomposed into $h_1(\lambda)+h_2(\lambda)$, where
\begin{align*}
&h_2(\lambda)=\max  \sum_{k=1}^p\lambda_kb^ky + dy &\\
&\text{Subject to: } y \in [0,1]^n,
\end{align*}
and $h_1(\lambda)$ is as defined in MLP1. Then the optimal objective function value of BQPC(p) can be identified by finding the global maximum of $h(\lambda)$ over all $\lambda \in R^p$ for which BLP1 is feasible. As a function of $\lambda$, $h_2(\lambda)$ is piecewise linear convex~\cite{gal} and $h_1(\lambda)$ is piecewise linear concave~\cite{gal}. Thus, $h(\lambda)$ is piecewise linear but need not be convex or concave. But $h_1(\lambda)$ is linear when $\lambda$ is restricted to any characteristic region of MLP1. Thus, $h(\lambda)$ is convex when $\lambda$ is restricted to  any characteristic region associated with $h_1(\lambda)$. Hence, the maximum of $h(\lambda)$ is attained at a breakpoint of $h_1(\lambda)$. Let $\lambda=\lambda^0$ be such a breakpoint  and  $x(\lambda^0)$ be an optimal basic feasible solution of MLP1 at $\lambda=\lambda^0$. By Theorem~\ref{zeroone}, $x(\lambda^0)$ is a 0-1 solution and it yields the optimal value $y(\lambda)$ as given by (\ref{opty}) for the $y$ variables for BQPC(p) subject to the condition that $x=x^0(\lambda)$. Since one of the $x(\lambda)$ for $\lambda \in S_1$ gives an optimal $x$-value, the result follows.
\end{proof}

Based on Theorem~\ref{opt}, BQPC(p) and hence BQP01 can be solved by generating the set $S_1$, computing the set $S_2=\{(x(\lambda),y(\lambda)) : \lambda \in S_1\}$ and choosing the best solution in $S_2$.  Note that we do not need to compute explicitly the set $S_1$. The solution set $\{x(\lambda) : \lambda \in S_1\}$ can be identified without computing $\lambda$. Let $\cl{(B,L,U)}$ be a dual feasible and dual non-degenerate basis structure. Let $\cl{B}=\{B1,B2,\ldots ,Bp\}$ and $B$ be the associated basis matrix. For each $\tau\in \{0,1\}^p$, we get an extreme point $\lambda^{\tau}=Bw + A^u\textbf{v}$ of its characteristic region. This follows from the inequality (\ref{feas}). But then, from equation (\ref{bv}), the corresponding basic variables $x(\lambda^{\tau})_B$ is precisely $\tau$. The non-basic variables are fixed at 0/1 values guided by $\cl{L}$ and $\cl{U}$. Thus, for each basis structure $\cl{(B,L,U)}$ we can generate $2^p$ basic feasible solutions corresponding to its extreme points by varying $\tau\in \{0,1,\}^p$ (see proof of Theorem~\ref{zeroone}).  For each such basic feasible solution $x$  we can compute the corresponding $y$ value easily (see equation~(\ref{opty})). Below we present a high-level description of our algorithm for solving BQPC(p): \\

\begin{tabular}{lp{13.5cm}}
\textbf{Step 1:} & Let $A$ be the coefficient matrix of MLP1 and $\Gamma$ be the collection of all dual feasible basis structures associated $A$.\\
\textbf{Step 2:} & For each basis structure $(\cl{B,L,U})\in \Gamma$, construct the set of optimal basic feasible solutions corresponding to the extreme points of its characteristic region. Let $\bar{S}$ be the collection of all such solutions obtained. \\
\textbf{Step 3:} & For each each $x\in \bar{S}$  compute the best $y\in R^n$, say $y^x$. Let $S_2=\{(x,y^x) : x \in \bar{S}\}$.\\
\textbf{Step 4:} & Output the best solution in $S_2$. \\
\end{tabular}
~\vspace{7pt}

 \noindent There are ${}^m\cl{C}_p$ choices for $\cl{B}$ and for each such choice, there is a unique allocation of non-basic variables to $\cl{L}$ and $\cl{U}$ at zero or one values. (The uniqueness follows from dual non-degeneracy assumption which can be achieved by appropriate perturbation of the cost vector.) The basis inverse can be obtained in $O(p^3)$ time and given this inverse, $\cl{L}$ and $\cl{U}$ can be identified in $O(mp^2)$ time so that $(\cl{B,L,U})$ is dual feasible. Thus $|\Gamma|\leq {}^m\cl{C}_p$ and $\Gamma$ can be identified in $O({}^m\cl{C}_p(p^3+mp^2)$ time.  The characteristic region associated with each $(\cl{B,L,U})\in \Gamma$ has at most $2^p$ extreme points  and the optimal solution to MLP1 when $\lambda$ is fixed at these extreme points can be identified as discussed in the proof of Theorem~\ref{zeroone} without explicitly computing $\lambda$.  Thus, given $\Gamma$, $\bar{S}$ in Step 2 can be identified in $O({}^m\cl{C}_p2^pm)$ time. For each $x \in \bar{S}$ we can compute the corresponding optimal $y^x$ in $O(mnp)$ time using equation (\ref{opty}). The best solution in $S_2$ can be identified in $O({}^m\cl{C}_p2^p(mnp))$ time. Summarizing the foregoing discussions, we have,

 \begin{theorem}\label{comp}BQP01 can be solved in $O({}^m\cl{C}_p2^pmnp)$ time when rank of $Q=p$ and $Q$ is given in the rank factored form.\end{theorem}
Note that for a fixed $p$ the above bound is polynomial and for $p=O(\log n)$ it is quasi-polynomial.  For specific choices of $p$, the complexity of the procedure discussed above may be improved. This is illustrated in the next subsection for the case when $p=1$.

\subsection{Rank one cost matrix}

 Theorem~\ref{comp} guarantees that if rank of $Q$ is 1, BQP01 can be solved in $O(m^2n)$ time. We now show that the problem can be solved in $O(n\log n)$ time by careful organization of our computations. Recall that $m\leq n$. As in the general case, let us consider the bilinear equivalent version:
\begin{flalign*}
\text{BQPC(1):\hspace{1cm} }&\text{Maximize } axby+cx+dy&\\
&\text{Subject to: }  x \in [0,1]^m, y\in [0,1]^n,&
\end{flalign*}
where $a=(a_1,a_2,\ldots ,a_m),c=(c_1,c_2,\ldots ,c_m) \in R^m$ and $b=(b_1,b_2,\ldots ,b_n),d=(d_1,d_2,\ldots ,d_n)\in R^n$. Let $A^-=\{i : a_i < 0\}$ and $A^+=\{i : a_i >0\}$. Define $\underline{\lambda}=\sum_{i\in A^-}a_i$ and $\overline{\lambda}=\sum_{i\in A^+}a_i$, where summation over the empty set is taken as zero.  Note that $\underline{\lambda}$ and $\overline{\lambda}$ are respectively the smallest and the largest values of $ax$ when  $x\in [0,1]^m$. Consider the {\it parametric continuous knapsack problem} (PKP($\lambda$)) given below.
\begin{flalign*}
\text{PKP$(\lambda)$:\hspace{1cm} }&\text{Maximize   } cx&\\
&\text{Subject to: } ax = \lambda&\\
 &\hspace{2.0cm}x \in [0,1]^m, \mbox{ and } \underline{\lambda} \leq \lambda \leq \overline{\lambda}.&
\end{flalign*}
This is a special case of MLP1 for $p=1$. Let $h_1(\lambda)$ be the optimal objective function value of PKP($\lambda$) for a given $\lambda$. Then for $\underline{\lambda} \leq \lambda \leq \overline{\lambda},$ $h_1(\lambda)$ is a piecewise linear concave function~\cite{murty}. Let $\underline{\lambda}=\lambda_1 < \lambda_2 < \cdots <\lambda_s=\overline{\lambda}$ be the breakpoints of $h_1(\lambda),$  $x^k$ be an optimal basic feasible solution of PKP($\lambda$) for $\lambda\in [\lambda_{k},\lambda_{k+1}], 1\leq k \leq s-1$, and $x^s$ be an optimal basic feasible solution to PKP($\lambda^s$). Let $y^k$ be an optimal solution to BQPC(1) when $x$ is restricted to $x^k$, the vector $y^k$ can be identified by appropriate modification of the equation (\ref{opty}). By Theorem~\ref{opt}, there exists an optimal solution to BQPC(1) amongst the solutions $x^k,y^k : k=1,2,\ldots ,s.$

From Lemma~\ref{lm7}, the number of breakpoints of $h_1(\lambda)$ is at most $2m$. We now observe that the number of breakpoints of $h_1(\lambda)$ cannot be more than $m+1$ and obtain closed form values of these breakpoints. \\

 Let $\displaystyle{T=\left\{\frac{c_i}{a_i} : i=1,2,\ldots ,m,\  a_i\neq 0\right \}}$ and consider a descending arrangement
 \begin{align}\label{daa}
 \frac{c_{\pi_1}}{a_{\pi_1}} > \frac{c_{\pi_2}}{a_{\pi_2}} > \cdots > \frac{c_{\pi_r}}{a_{\pi_r}}
 \end{align}
  of all distinct elements of $T$. Let $\displaystyle{T(k) = \left\{i : \frac{c_{\pi_k}}{a_{\pi_k}}=\frac{c_i}{a_i}\right\}}$. Then the breakpoints of $h_1(\lambda)$ are given by
  \begin{align*}
  \lambda_1=\underline{\lambda}\mbox{ and }\lambda_{k+1}=\lambda_{k}+\sum_{i\in T(k)}|a_i| \mbox{ for } k=1,2,\ldots, r.
  \end{align*}

An optimal solution to PKP($\lambda$) at $\lambda=\lambda_k$ for $k=1,2,\ldots ,r+1$  can be identified recursively as

\begin{minipage}{6.0cm}
\begin{align*}
x^1_i = \begin{cases} 1 &\mbox{ if } a_i=0 \mbox{ and } c_i > 0 \mbox{ or } a_i < 0 \\
0 & \mbox{ otherwise }
\end{cases}
\end{align*}
\end{minipage}
\begin{minipage}{6.0cm}
\begin{align*}
\mbox{ and }\;\;x^{k+1}_i = \begin{cases} x^{k}_i &\mbox{ if } i\notin T(k) \\
1 & \mbox{ if } i\in T(k) \mbox{ and } a_i > 0\\
0 & \mbox{ otherwise}.
\end{cases}
\end{align*}
\end{minipage}


 Thus, it can be verified that given $h(\lambda_k)$ and $x^k$, $h(\lambda_{k+1})$ and $x^{k+1}$ can be identified in $O(|T(k)|)$ time. The complexity for generating these solutions and breakpoints are dominated by that of constructing the descending arrangement (\ref{daa}) which is $O(m\log m)$. Note that $h_1(\lambda)$ has at most $m+1$ breakpoints and given $x^{k}$, a corresponding solution $y^{k}$  can be computed in $O(n)$ time. This leads to a complexity of $O(mn)$. The bottleneck operation here is the computation of $y^k$ for $k=1,2,\ldots ,r$. We now show that these points can be identified in $O(n\log n)$ time.

 Consider the {\it parametric unconstrained linear optimization problem}
\begin{flalign*}
\text{ULP$(\mu)$: \hspace{1cm} }&\text{Maximize   } dy+\mu by&\\
&\text{Subject to: } y \in [0,1]^n\mbox{ and } \underline{\lambda} \leq \mu \leq \overline{\lambda}.&
\end{flalign*}

\noindent Let $h_2(\mu)$ be the optimal objective function value of ULP($\mu$). Then $h_2(\mu)$ is a piecewise linear convex function.

Let $S^+=\{j : d_j + \underline{\lambda} b_j \geq 0\}$,  $S^-=\{j : d_j + \underline{\lambda} b_j < 0\}$, $B^+=\{j : b_j > 0\}$, and $B^-=\{j : b_j < 0\}$. Also, let $T_2=\{\minus~\frac{d_j}{b_j} : j \in B^+\cup B^-\}$ and consider an ascending arrangement

\begin{align}\label{ar}
 \underline{\lambda} < \minus~\frac{d_{\tau_1}}{b_{\tau_1}} < \minus~\frac{d_{\tau_2}}{b_{\tau_2}} \cdots <\minus~\frac{d_{\tau_t}}{b_{\tau_t}}
 \end{align}

\noindent of all distinct elements of $T_2$ greater than $\underline{\lambda}$.
 Let $\mu_1,\mu_2, \ldots ,\mu_t$  be the breakpoints of $h_2(\mu)$. Then $\mu_{\ell}=\minus~\frac{d_{\tau_{\ell}}}{b_{\tau_{\ell}}} \mbox{ for } \ell=1,2,\dots t.$
Let $\Delta_i = \{j :  \frac{d_j}{b_j}=\frac{d_{\tau_i}}{b_{\tau_i}}\}$. Then the optimal solution $y_{\ell}$ for ULP($\mu$) corresponding to the breakpoint $\mu_{\ell}$ for $\ell = 1,2,\ldots ,t$ is given recursively by

\begin{minipage}{6.0cm}
 \begin{align*}
y_j^{\ell} = \begin{cases} y^{\ell-1}_j &\mbox{ if } j\notin \Delta_{\ell} \\
1 & \mbox{ if } j\in \Delta_{\ell} \mbox{ and } y^{\ell-1}_j = 0\\
0 & \mbox{ if  } j\in \Delta_{\ell} \mbox{ and } y^{\ell-1}_j = 1,
\end{cases}
\end{align*}

\end{minipage}
\begin{minipage}{6.0cm}
\begin{align*}
\text{where }y_j^{0} = \begin{cases} 1 &\mbox{if }  j\in S^+ \\
0 & \mbox{otherwise. }
\end{cases}
\end{align*}
\end{minipage}

Define \begin{align*}
D^0&=\sum_{j\in S^+}d_j,\; B^0 =\sum_{j\in S^+}b_j,
D^{\ell}=D^{\ell-1}-\sum_{j\in \Delta_{\ell},y_j^{\ell-1}=1}d_j+\sum_{j\in \Delta_{\ell},y_j^{\ell-1}=0}d_j\mbox{ and }\\
B^{\ell}&=B^{\ell-1}-\sum_{j\in \Delta_{\ell},y_j^{\ell-1}=1}b_j+\sum_{j\in \Delta_{\ell},y_j^{\ell-1}=0}b_j.
\end{align*}

\noindent Then the optimal objective function value at $\mu_{\ell}$ is given by
$h_2(\mu_{\ell})=D^{\ell}+\mu_{\ell}B^{\ell}.$

Given $y^{\ell-1}$, $D^{\ell-1}$ and $B^{\ell-1},$ we can compute $y^{\ell}$, $D^{\ell}$, and $B^{\ell}$ in $O(|\Delta_{\ell}|)$ time and, hence,  $h_2(\mu_{\ell})$ and $y^{\ell}$  can be identified in $O(|\Delta_{\ell}|)$ time. Since $\Delta_{\ell}\cap \Delta_k=\emptyset$ for $\ell\neq k$, $y^{\ell}$ and $h_2(\mu^{\ell})$ for $\ell=1,2,\ldots ,t$ can be identified in $O(n\log n)$ time.\\

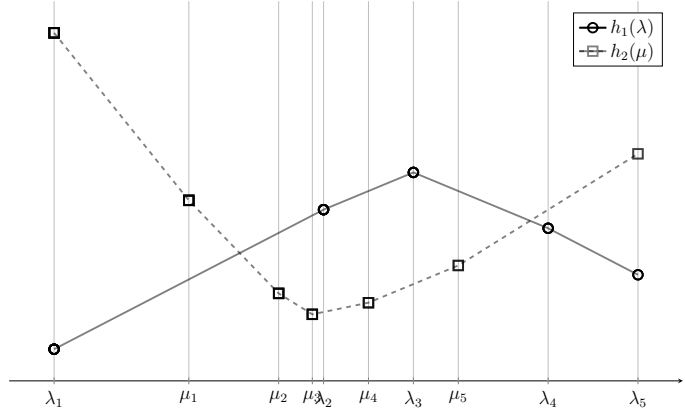
\begin{figure}
\centering
\begin{tikzpicture}[thick,scale=0.6, every node/.style={transform shape}]
	\begin{axis}[
		xmin=-6,
		xmax=9,
		width=\textwidth,
		height=10cm,
		legend pos=north east,
		xlabel={},
		ylabel={objective},
		xtick=\empty,
		ytick=\empty,
		hide y axis=true,
		axis x line=bottom,
		extra x ticks={-5, -2, 0, 0.75, 1, 2, 3, 4, 6, 8},
		extra x tick labels={$\lambda_1$, $\mu_1$, $\mu_2$, $\mu_3$, $\lambda_2$, $\mu_4$, $\lambda_3$, $\mu_5$, $\lambda_4$, $\lambda_5$},
		extra x tick style={grid=major, tick label style={anchor=north}},
		cycle list={
			{black,very thick,opacity=0.5,mark=o,mark size=3pt,mark options={opacity=1,fill=none}},
			{black,dashed,very thick,opacity=0.5,mark=square,mark size=3pt,mark options={solid}},
		},
	]
	\addplot+ coordinates {
		(-5, 11)
		(1, 26) +- (0,16)
		(3, 30) +- (0,20)
		(6, 24) +- (0,14)
		(8, 19) +- (0,9)
	};
	\addlegendentry{$h_1(\lambda)$}
	\addplot+ coordinates {
		(-5, 45)
		(-2, 27)
		(0, 17)
		(0.75, 14.75)
		(2, 16)
		(4, 20)
		(8, 32)
	};
	\addlegendentry{$h_2(\mu)$}
	\end{axis}
\end{tikzpicture}
\caption{An example of $h_1(\lambda)$ and $h_2(\mu)$ when $a$ = [2, 2, -3, 4, -2], $c$ = [4, 5, 6, 10, 5], $b$ = [1, 1, -4, 0, -1, -2, 1], $d$ = [5, -2, 3, 3, 4, 0, 2].}\label{fig11} \end{figure}

Now the algorithm for solving BQPC(1) can be described as follows.
First, compute $x^1,y^1, h_1(\underline{\lambda})$ and $h_2(\underline{\lambda})$. Set $f(x^1,y^1)=h_1(\underline{\lambda})+h_2(\underline{\lambda})$. Sort all breakpoints of $h_1(\lambda)$ and $h_2(\mu)$ for $\underline{\lambda} \leq \lambda, \mu \leq \overline{\lambda}$ (see Fig:~\ref{fig11}) and scan these breakpoints starting from $\underline{\lambda}$ in the increasing order. As we pass breakpoints of $h_2(\mu)$ keep updating the solution $y$ of ULP($\mu$) corresponding to this breakpoint and the objective function value of this solution until we hit a breakpoint $\lambda_k$ of $h_1(\lambda)$. At this point compute the solution $x^k$ and $h_1(\lambda_k)$. The most recent solution $y$ identified is selected as $y^k$ and compute $h_2(\lambda_k)$. Note that $h_2(\lambda_k)$ can be obtained in $O(1)$ time using slope of $h_2(\mu)$ for the interval containing $\lambda^k$. Update $f(x^k,y^k)$ and the process is continued until all breakpoints of $h_1(\lambda)$ including $\overline{\lambda}$ are examined and the overall best solution is selected. It is not difficult to verify that the complexity of this procedure is $O(n\log n)$.

Let us now consider a special case of BQP01 when $Q$ is of rank one (equivalently a special case of BQPC(1)) of the form
\begin{flalign*}
\text{BQPC($1,c\vee d =0$): \hspace{1cm}}&\text{Maximize } (a_0+ax)(b_0+by)+cx+dy&\\
&\text{Subject to: }  x \in [0,1]^m, y\in [0,1]^n,&
\end{flalign*}
where either $c=0$ or $d=0$.  We now show that this problem can be solved more efficiently.
\begin{theorem}An optimal solution to BQPC($1,c\vee d =0$) can be identified in O($n$) time.
\end{theorem}
\begin{proof}Suppose $d=0$.  Let $L(\lambda)=\text{Maximum }_{x\in \{0,1\}^m}\left\{ \left(a_0+ax\right)\lambda +cx\right\}$. Clearly, $L(\lambda)$ is
is a piecewise linear convex function of $\lambda$. Suppose  $b_0+by$ maximizes at $y^0$ and minimizes at $y^*$ with respective objective function values, say, $\lambda^0$ and $\lambda^*$. As $\lambda$ varies in the interval $[\lambda^*,\lambda^0]$, $L(\lambda)$ traces the best objective function values for BQPC($1,c\vee d =0$) for all possible solution vectors $y$. Thus, BQPC($1,c\vee d =0$) reduces to maximizing $L(\lambda)$ for
$\lambda \in [\lambda^*,\lambda^0]$. Convexity of $L(\lambda)$ guarantees that its maximum is attained when $\lambda= \lambda^0$ or $\lambda^*$, i.e., when $y=y^0$ or $y=y^*$. But when $y$ is fixed, an optimal $x$ can be obtained in linear time. Since $y^0$ and $y^*$ can be identified in $O(n)$ time, the result follows. The case when $c=0$ can be established analogously.
\end{proof}

\subsection{Additively decomposable cost matrix}

Let us now examine the case when $q_{ij}=a_i+b_j$ for $i=1,2,\ldots ,m$ and $j=1,2,\ldots ,n$.   Note that when $q_{ij}=a_i+b_j$, rank of $Q$ is at most 2 and hence can be solved in polynomial time. We now present an $O(mn\log n)$ algorithm to solve this problem.
\begin{theorem}If $q_{ij}=a_i+b_j$ for $i=1,2,\ldots ,m$ and $j=1,2,\ldots ,n$, then BQP01 can be solved in $O(mn\log n)$.\end{theorem}
\begin{proof}
For any feasible solution $x,y$,
\begin{align*}
f(x,y)&=  x^TQy + cx+dy+c_0\\
&=\sum_{i=1}^ma_ix_i\sum_{j=1}^ny_j+\sum_{j=1}^nb_jy_j\sum_{i=1}^mx_i + \sum_{i=1}^mc_ix_i+\sum_{j=1}^nd_jy_j+c_0.
\end{align*}
Let $\sum_{j=1}^ny_j = K$ and $\sum_{i=1}^mx_i = L$, where $K$ and $L$ are two parameters. Then,
\begin{align*}
f(x,y)& = \sum_{i=1}^m(Ka_i+c_i)x_i+\sum_{j=1}^n(Lb_j+d_j)y_j +c_0.
\end{align*}
Consider the optimization problem
\begin{flalign*}
\text{ILP(K,L): \hspace{1cm}}&\text{Maximize}  \sum_{i=1}^m(Ka_i+c_i)x_i+\sum_{j=1}^n(Lb_j+d_j)y_j +c_0&\\
&\text{Subject to } \sum_{j=1}^ny_j = K \text{ and } \sum_{i=1}^mx_i = L&
\end{flalign*}

For $K = 0, 1, \ldots, n$, let $\alpha^K$ be a permutation of size $m$ such that $K a_{\alpha^K(i)} + c_{\alpha^K(i)} \ge K a_{\alpha^K(i + 1)} + c_{\alpha^K(i + 1)},$  $i=1,\ldots ,m-1$.  For $L = 0, 1, \ldots, m$, let $\beta^L$ be a permutation of size $n$ such that $L b_{\beta^L(j)} + d_{\beta^L(j)} \ge L b_{\beta^L(j + 1)} + d_{\beta^L(j + 1)}$, $j=1,2,\ldots ,n-1$.
Observe that, for fixed $K$ and $L$, the optimal $x = x^{K, L}$ can be obtained by setting $x_i = 1$ for $i = \alpha^K(1), \alpha^K(2), \ldots, \alpha^K(L)$ and $x_i = 0$ for the rest of indices.
Similarly, the optimal $y = y^{K, L}$ can be obtained by setting $y_j = 1$ for $j = \beta^L(1), \beta^L(2), \ldots, \beta^L(K)$ and $y_j = 0$ for the rest of indices.

Let $f^{K, L}_1 = \sum_{i = 1}^m (K a_i + c_i) x^{K, L}_i$ and $f^{K,
L}_2 = \sum_{j = 1}^n (L b_j + d_j) y^{K, L}_j$.  Assume that we know the value of $f^{K, L}_1$ for some $K$ and $L$.  Then we can calculate $f^{K,L + 1}_1$ in $O(1)$ time as:
$$
f^{K,L + 1}_1 = f^{K, L}_1 + K a_{\alpha^K(L + 1)} + c_{\alpha^K(L + 1)} \,.
$$
Observe also that $f^{K, 0}_1 = 0$.  Hence, for a fixed $K$, we can calculate $f^{K, L}_1$ for each $L \in \{ 0, 1, \ldots, m \}$ in $O(m)$ time.  Similarly, for a fixed $L$, we can calculate $f_2^{K, L}$ for each $K = 0, 1, \ldots, n$ in $O(n)$ time.  Let $(K^0, L^0)$ be the values of $K$ and $L$ that maximize $f^{K, L}_1 + f^{K, L}_2$.
Then the optimal solution of the BQP01 is $(x^{K^0, L^0}, y^{K^0, L^0})$, and it can be obtained in $O(nm \log m + mn \log n + mn) = O(mn \log n)$ time.
\end{proof}

\section{Conclusion}
In this paper we studied the problem BQP01 which generalizes QP01, a well studied combinatorial optimization problem. BQP01 is known to be MAX SNP hard. Several interesting polynomially solvable special cases of the problem are identified. In particular, we showed that when rank of the matrix $Q$ is fixed, BQP01 can be solved in polynomial time and improved complexity results are provided for the rank one case and a special case when rank of $Q$ is at most two. By restricting $m=O(\log n)$ or by restricting the size of minimum negative eliminator of $Q$ to be $O(\log n)$, we obtained additional polynomially solvable cases. If $m=\sqrt[k]{n}$ BQP01 is MAX SNP hard and a similar result is obtained if the size of a minimum negative eliminator of $Q$ is $O(\sqrt[k]{n})$. It would be interesting to explore the complexity of the problem when the problem size or data restrictions fall in between these extreme cases. Exploiting the algorithms for our polynomially solvable special cases, efficient exact and heuristic algorithms may be obtained to solve BQP01. This is a topic for further investigation.\\

\noindent\textbf{Acknowledgement:} We are thankful to the referees for their insightful comments which improved the presentation of the paper.

\end{document}